\documentclass[12pt, oneside]{scrartcl}
\usepackage[english]{babel}
\usepackage[T1]{fontenc}
\usepackage{amssymb}
\usepackage{amsmath}
\usepackage{hhline}
\usepackage{longtable}
\usepackage{amscd}
\usepackage{array}
\usepackage{delarray}
\usepackage{multicol}
\usepackage{makeidx}

\usepackage{float}

\usepackage{tikz}
\usepackage{calc}
\usepackage{hyperref}
\usepackage{amsthm}
\usepackage{cleveref}
\usetikzlibrary{angles,quotes}
\usetikzlibrary{calc}
\usetikzlibrary{patterns}
\usetikzlibrary{graphs}
\usetikzlibrary{graphs.standard}
\usetikzlibrary{decorations.markings,shapes.geometric,positioning}

\newtheorem{theorem}{Theorem}

\newtheorem{conjecture}{Conjecture}

\newtheorem{lemma}{Lemma}
\newtheorem{remark}{Remark}
\newtheorem{claim}{Claim}[section]

\newtheorem{problem}{Problem}
\newtheorem{proposition}{Proposition}

\usepackage{tikz}
\author{ Maidoun Mortada \footnote{Corresponding author. KALMA Laboratory, Faculty of Sciences, Lebanese University, Baalbek, Lebanon; Graph Theory and Operation Research, Department of Mathematics and Physics, Lebanese International University (LIU), Beirut, Lebanon; and Basic and
Applied Sciences Research, Al Maaref University, Beirut, Lebanon. Email: maydoun.mortada@liu.edu.lb} \hspace{0.4cm} Ayman El Zein \footnote{Computer Science Department, University of Sciences and Arts in Lebanon, Beirut, Lebanon. Email: a.elzein@usal.edu.lb} \\
Sara Al Hajjar \footnote{
KALMA Laboratory, Lebanese University, Beirut, Lebanon \\
Univ. Bordeaux, CNRS, Bordeaux INP, LaBRI, UMR 5800, F-33400, Talence,
France. Email: Sara.bhajjar@gmail.com.}
}
\date{}
    
\begin{document} 
\title{On $(1,1,2,3)$- and $(1,1,3,3,3)$-Packing Colorings of Claw-Free Subcubic Graphs}
\maketitle

\begin{abstract}
For a non-decreasing sequence $S=(a_1,a_2,\ldots,a_r)$ of positive integers, an $S$-packing coloring of a graph $G$ is a partition of $V(G)$ into sets $A_1,\ldots,A_r$ such that any two distinct vertices in $A_i$ are at distance greater than $a_i$, for every $i\in\{1,\ldots,r\}$. Gastineau and Togni [\emph{Discrete Math.} 339 (2016), 2461--2470] asked whether every subcubic graph, except the Petersen graph, is $(1,1,2,3)$-packing colorable. In this paper, we prove that every claw-free subcubic graph is $(1,1,2,3)$-packing colorable. Moreover, we show that every connected claw-free subcubic graph, except a single graph $\mathcal{H}$, is $(1,1,3,3,3)$-packing colorable, thereby confirming a conjecture of the first two authors. Both results are best possible. Our proofs rely on a structural framework based on the skeleton and core graphs of a claw-free subcubic graph, together with a Hall-type matching argument that reduces the construction of suitable $3$-packings to a matching problem in an auxiliary bipartite graph.

\end{abstract}

\noindent\textbf{Mathematics Subject Classification:} 05C15\\
\textbf{Keywords}: graph coloring; packing coloring; claw-free; cubic graph; subcubic graph; Hall's theorem.

\section{Introduction}

Packing coloring is a distance-constrained variant of graph coloring. Given a non-decreasing sequence $S=(a_1,a_2,\ldots,a_r)$ of positive integers, an $S$-packing coloring of a graph $G$ is a partition of $V(G)$ into sets $A_1,\ldots,A_r$ such that any two distinct vertices in $A_i$ are at distance greater than $a_i$, for every $i \in \{1,\ldots,r\}$. When $S=(1,2,\ldots,r)$, this notion coincides with the classical packing coloring introduced by Goddard \emph{et al.}~\cite{d}, and the minimum value of $r$ for which such a coloring exists is called the packing chromatic number of $G$. Since its introduction, packing coloring and its generalization to $S$-packing coloring have received considerable attention, especially for graphs of bounded maximum degree and, in particular, for subcubic graphs; see, for instance, \cite{bkl,3,bf,11,8,9,AM1,AM2,16,outer,24,m,LW,LZZ24,AM3,26,27,mai,mai1,i,a,ZZ2026} and the survey by Bre\v{s}ar \emph{et al.}~\cite{bfk}.

Among the numerous questions concerning packing colorings of subcubic graphs, one of the most influential was posed by Gastineau and Togni~\cite{16}, who asked whether every subcubic graph, except the Petersen graph, admits a $(1,1,2,3)$-packing coloring. This problem has stimulated a series of works devoted to improving the known packing colorings of subcubic graphs. In particular, Liu, Zhang, and Zhang~\cite{LZZ24} proved that every subcubic graph is $(1,1,2,2,3)$-packing colorable, providing the first general result using only five colors. Later, El Zein and Mortada~\cite{AM1} showed that every non-regular subcubic graph is $(1,1,2,2)$-packing colorable and, more generally, that every subcubic graph admits a $(1,1,2,2,k)$-packing coloring for every integer $k\geq3$, thereby extending and strengthening several previously known results. Most recently, Hou, Liu, and Wang~\cite{HLW} established that every connected subcubic graph, except the Petersen graph, is $(1,1,2,2)$-packing colorable, completely settling the corresponding $(1,1,2,2)$-packing coloring problem. Nevertheless, the original conjecture of Gastineau and Togni for the stronger sequence $(1,1,2,3)$ remains open. This naturally raises the question of whether the conjecture holds for important subclasses of subcubic graphs.

The present paper is motivated by this question. We investigate claw-free subcubic graphs, an important subclasse of subcubic graphs, and prove that every claw-free subcubic graph admits a $(1,1,2,3)$-packing coloring. Thus, our result confirms the original conjectured behavior for this natural graph class. We also prove a stronger structural result by showing that every claw-free subcubic graph, with the exception of a single graph $\mathcal{H}$ (see Figure \ref{figure_1}), admits a $(1,1,3,3,3)$-packing coloring. This confirms the conjecture proposed by the first two authors \cite{AM4}. Both results are best possible, showing that the obtained packing colorings cannot, in general, be further improved.

Our proofs are based on a structural and combinatorial framework that reduces the original distance-coloring problem to a sequence of simpler selection problems. Starting from a claw-free subcubic graph $G$, we first construct its skeleton $S_G$ by suppressing the paths whose internal vertices do not belong to triangles. This operation removes structurally inessential vertices while preserving the distance relations needed for the packing conditions. The skeleton is naturally decomposed into triangle and diamond blocks. By contracting these blocks, we obtain the core graph $C_G$, which records how the local triangle structures of $G$ interact at a global level. In this way, the skeleton captures the relevant local distances, whereas the core captures the global organization of the graph.

The main advantage of this framework is that it transforms the construction of suitable $3$-packings into a matching problem. Given an independent set $I$ of the core graph, we define an auxiliary graph $H_I$ on the vertices of the corresponding blocks of the skeleton. Two vertices are adjacent in $H_I$ precisely when choosing them simultaneously would violate the $3$-packing condition. We prove that every component of $H_I$ has order at most three. We then construct a bipartite graph $B_I$ whose left part represents the blocks indexed by $I$, and whose right part represents the components of $H_I$. Selecting one suitable vertex from each block is therefore equivalent to finding a matching that saturates the left part of $B_I$. Hall's theorem guarantees such a matching and consequently produces a $3$-packing meeting every required block. Thus, the skeleton--core decomposition, together with the auxiliary conflict graph and the Hall-type matching argument, provides a systematic method for converting distance constraints into a tractable combinatorial selection problem.
Beyond the results established in this paper, we believe that the proposed skeleton--core framework provides a versatile approach for studying packing colorings and related distance-constrained coloring problems in claw-free graphs and other graph classes.

The paper is organized as follows. In Section~2, we recall the main tools used in the proofs and introduce the skeleton and core graphs associated with a claw-free subcubic graph. In Section~3, we define the auxiliary graphs associated with an independent set of the core graph and use Hall's theorem to construct a $3$-packing meeting every corresponding block of the skeleton. Section~4 is devoted to the proof that every claw-free subcubic graph is $(1,1,2,3)$-packing colorable. In Section~5, we adapt the skeleton--core and matching framework to prove that every connected claw-free subcubic graph, except $\mathcal{H}$, is $(1,1,3,3,3)$-packing colorable. Finally, in Section~6, we discuss the sharpness of both results and propose a related open problem. We conclude with remarks and open problems.

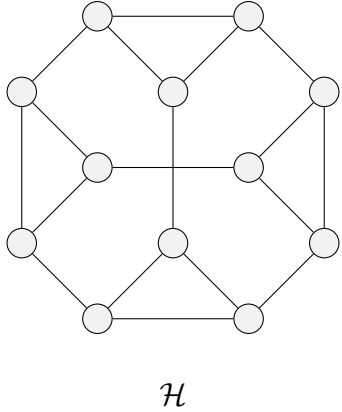
\begin{figure}[h!]
\centering
\begin{tikzpicture}[
    node/.style={circle, draw, fill=gray!10, minimum size=1mm},
    edge/.style={thick},
    green edge/.style={thick, draw=green},
    red edge/.style={thick, draw=red},
    black edge/.style={thick, draw=black},
    dotted edge/.style={thick, dotted},
    weight/.style={font=\bfseries\small}
]


\node[node] (1) at (6,0) {};
\node[node] (2) at (8,0) {};
\node[node] (3) at (9,1) {};
\node[node] (4) at (9,3) {};
\node[node] (5) at (8,4) {};
\node[node] (6) at (6,4) {};
\node[node] (7) at (5,3) {};
\node[node] (8) at (5,1) {};

\node[node] (9) at (7,1) {};
\node[node] (10) at (8,2) {};
\node[node] (11) at (7,3) {};
\node[node] (12) at (6,2) {};

\draw (1) -- (2);
\draw (2) -- (3);
\draw (3) -- (4);
\draw (4) -- (5);
\draw (5) -- (6);
\draw (6) -- (7);
\draw (7) -- (8);
\draw (8) -- (1);
\draw (1) -- (9);
\draw (2) -- (9);
\draw (3) -- (10);
\draw (4) -- (10);
\draw (5) -- (11);
\draw (6) -- (11);
\draw (7) -- (12);
\draw (8) -- (12);
\draw (9) -- (11);
\draw (10) -- (12);

\node at (7, -1) {$\mathcal{H}$};

\end{tikzpicture}
\caption{The graph $\mathcal{H}$ that is the sole exception of the $(1,1,3,3,3)$-packing colorability of connected claw-free subcubic graphs.}\label{figure_1}
\end{figure}

\section{Preliminaries}
We will use Brooks' theorem and Hall's theorem.

\begin{theorem}[Brooks' theorem]\label{t1}
Let $G$ be a connected graph with maximum degree $\Delta$. Then $\chi(G)\leq \Delta,$ unless $G$ is a complete graph or an odd cycle. Equivalently,
$\chi(G)=\Delta+1$ if and only if $G$ is a complete graph or an odd cycle.
\end{theorem}

\begin{theorem}[Hall's theorem]\label{t2}
Let $G$ be a bipartite graph with bipartition $(X,Y)$. Then, $G$ has a matching that saturates every vertex of $X$ if and only if
$|N(A)|\ge |A|$ for every subset $A\subseteq X$.
\end{theorem}

Throughout this section, we consider connected claw-free subcubic graphs with minimum degree at least $2$ and containing a $3$-vertex. This assumption causes no loss of generality for the proofs of our main results: cycles clearly admit the required $S$-packing colorings, while any pendant path can be colored after the remaining graph has been colored. In particular, every graph considered below contains a triangle, since every $3$-vertex in a claw-free graph belongs to a triangle.

Let $G$ be a graph as above such that $G\neq K_4$. A vertex of $G$ that is not contained in a triangle is called a \textit{bad vertex}. Since every $3$-vertex in a claw-free graph belongs to a triangle, every bad vertex has degree $2$. A \textit{bad path} is a path whose internal vertices are bad and whose end vertices are non-bad. The \textit{skeleton} $S_G$ of $G$ is the graph obtained by replacing every bad path by a single edge joining its end vertices (see Figure \ref{Figure_example}). Observe that the triangles in $G$ are those in $S_G$. Moreover, every vertex in $S_G$ is contained in a triangle and no more than two triangles. A vertex in $S_G$ is said to be \textit{heavy} if it is contained in two triangles, and \textit{light} otherwise. For a heavy (resp., light) vertex $x$ in $S_G$, we denote by $R_x^1$ and $R_x^2$ (resp., $R_x$) the triangles (resp., triangle) in $S_G$ that contain (resp., contains) $x$.

\begin{figure}[h!]
\centering
\begin{tikzpicture}[
    node/.style={circle, draw, fill=gray!10, minimum size=1mm},
    snode/.style={rectangle, draw, fill=gray!10, minimum size=1mm},
    edge/.style={thick},
    green edge/.style={thick, draw=green},
    red edge/.style={thick, draw=red},
    black edge/.style={thick, draw=black},
    dotted edge/.style={thick, dotted},
    weight/.style={font=\bfseries\small}
]


\node[node] (v11) at (0,0) {};
\node[node] (v12) at (0.5,0.8) {};
\node[node] (v13) at (-0.5,0.8) {};

\node[node,color={red},fill=red!10] (a1) at (0.7,0) {};
\node[node,color={red},fill=red!10] (a2) at (1.4,0) {};
\node[node,color={red},fill=red!10] (a3) at (2.1,0) {};

\node[node] (v21) at (2.8,0) {};
\node[node] (v22) at (3.6,-0.5) {};
\node[node] (v23) at (3.6,0.5) {};

\node[node,color={red},fill=red!10] (b1) at (3.6,1.2) {};
\node[node,color={red},fill=red!10] (b2) at (3.6,1.9) {};

\node[node] (v31) at (3.6,2.7) {};
\node[node] (v32) at (2.8,3.2) {};
\node[node] (v33) at (2,2.7) {};
\node[node] (v34) at (2.8,2.2) {};

\node[node,color={red},fill=red!10] (c1) at (1,2.7) {};

\node[node] (v41) at (0,2.7) {};
\node[node] (v42) at (-0.5,1.9) {};
\node[node] (v43) at (0.5,1.9) {};

\node[node] (v51) at (3.6,-1.2) {};
\node[node] (v52) at (2.8,-1.7) {};
\node[node] (v53) at (3.6,-2.2) {};

\draw (v11) -- (v12);
\draw (v12) -- (v13);
\draw (v13) -- (v11);

\draw (v21) -- (v22);
\draw (v22) -- (v23);
\draw (v23) -- (v21);

\draw (v31) -- (v32);
\draw (v32) -- (v33);
\draw (v33) -- (v34);
\draw (v34) -- (v31);
\draw (v34) -- (v32);

\draw (v41) -- (v42);
\draw (v42) -- (v43);
\draw (v43) -- (v41);

\draw (v51) -- (v52);
\draw (v52) -- (v53);
\draw (v53) -- (v51);

\draw (v11) -- (a1);
\draw (a1) -- (a2);
\draw (a2) -- (a3);
\draw (a3) -- (v21);

\draw (v23) -- (b1);
\draw (b1) -- (b2);
\draw (b2) -- (v31);

\draw (v33) -- (c1);
\draw (c1) -- (v41);

\draw (v42) -- (v13);
\draw (v43) -- (v12);

\draw (v22) -- (v51);

\node at (2, -2.5) {$G$};


\begin{scope}[xshift=6.5cm]

\node[node] (v11) at (0,0) {};
\node[node] (v12) at (0.5,0.8) {};
\node[node] (v13) at (-0.5,0.8) {};

\node[node] (v21) at (2.8,0) {};
\node[node] (v22) at (3.6,-0.5) {};
\node[node] (v23) at (3.6,0.5) {};

\node[node] (v31) at (3.6,2.7) {};
\node[node] (v32) at (2.8,3.2) {};
\node[node] (v33) at (2,2.7) {};
\node[node] (v34) at (2.8,2.2) {};

\node[node] (v41) at (0,2.7) {};
\node[node] (v42) at (-0.5,1.9) {};
\node[node] (v43) at (0.5,1.9) {};

\node[node] (v51) at (3.6,-1.2) {};
\node[node] (v52) at (2.8,-1.7) {};
\node[node] (v53) at (3.6,-2.2) {};

\draw (v11) -- (v12);
\draw (v12) -- (v13);
\draw (v13) -- (v11);

\draw (v21) -- (v22);
\draw (v22) -- (v23);
\draw (v23) -- (v21);

\draw (v31) -- (v32);
\draw (v32) -- (v33);
\draw (v33) -- (v34);
\draw (v34) -- (v31);
\draw (v34) -- (v32);

\draw (v41) -- (v42);
\draw (v42) -- (v43);
\draw (v43) -- (v41);

\draw (v51) -- (v52);
\draw (v52) -- (v53);
\draw (v53) -- (v51);

\draw (v11) -- (v21);

\draw (v23) -- (v31);

\draw (v33) -- (v41);

\draw (v42) -- (v13);
\draw (v43) -- (v12);

\draw (v22) -- (v51);

\node at (2,-2.5) {$S_G$};

\end{scope}

\begin{scope}[xshift=12cm]
    \node[snode] (v1) at (0,0){};
    \node[snode] (v2) at (1,0){};
    \node[snode] (v3) at (1,1){};
    \node[snode] (v4) at (0,1){};
    \node[snode] (v5) at (1,-1){};

    \draw (v1) -- (v2);
    \draw (v2) -- (v3);
    \draw (v3) -- (v4);
    \draw (v4) -- (v1);
    \draw (v2) -- (v5);

    \node at (0.5,-2.5) {$C_G$};
\end{scope}

\end{tikzpicture}
\caption{A claw-free subcubic graph \(G\), its skeleton \(S_G\), obtained by replacing each bad path (whose internal vertices are shown in red) with an edge, and the corresponding core \(C_G\), obtained by contracting every triangle or diamond of \(S_G\) into a single vertex.}\label{Figure_example}
\end{figure}
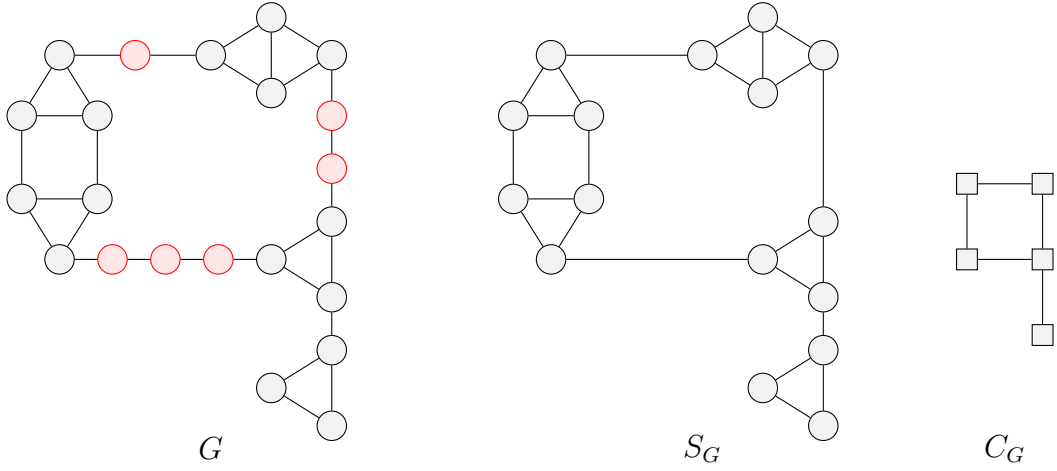

\begin{remark}\label{r1}
    For every two vertices $x,y\in S_G$, we have $d_{G}(x,y)\geq d_{S_G}(x,y)$.
\end{remark}

Observe that the vertex set of $S_G$ can be partitioned into $V_1,\dots ,V_k$ such that $S_G[V_i]$ is either a triangle or a diamond ($K_4-e$). The \textit{core} graph of $G$, denoted by $C_G$, is the graph obtained from $S_G$ by contracting every set of vertices among $V_1,\dots ,V_k$ to a vertex (see Figure \ref{Figure_example}). That is, $V(C_G)=\{v_1,\dots ,v_k\}$ and $v_iv_j\in E(C_G)$ if and only if $i\neq j$ and there exist $a\in V_i$ and $b\in V_j$ such that $ab\in E(S_G)$.

\begin{remark}\label{r2}
    Let $i\in \{1,\dots ,k\}$ and $x\in V_i$. As $S_G$ is subcubic, the following hold:\begin{itemize}
        \item[(i)] If $x$ is heavy, then $x$ has no neighbors outside $V_i$;
        \item[(ii)] If $x$ is a light $2$-vertex, then $x$ has no neighbors outside $V_i$; and
        \item[(iii)] If $x$ is a light $3$-vertex, then $x$ has a unique neighbor outside $V_i$.
    \end{itemize}
\end{remark}

Let $R$ and $S$ be two triangles in $S_G$. A pair $(R,S)$ is said to be an \textit{adjacent pair} if $R$ and $S$ have two common vertices; that is, there exist two heavy vertices $x$ and $y$ such that $R_x^1=R_y^1=R$ and $R_x^2=R_y^2=S$. Otherwise, $(R,S)$ is said to be a \emph{non-adjacent pair}. Let $(R,S)$ be a non-adjacent pair. We say that $(R,S)$ is a \emph{close pair} if a vertex of $R$ is adjacent to a vertex of $S$.

\section{A good $3$-packing}
The main goal of this section is to construct a $3$-packing that intersects every block corresponding to an independent set of the core graph. We begin by establishing structural properties of such independent sets in the skeleton graph. These properties allow us to transform the problem of selecting one suitable representative from each block into a matching problem in an auxiliary bipartite graph, which is then solved using Hall's theorem.
Throughout this section, let $G$ be a connected claw-free subcubic graph such that $G\neq K_4$, and let $S_G$ and $C_G$ denote its skeleton graph and its core graph, respectively. Recall that the vertex set of $S_G$ is partitioned into sets $V_1,\ldots,V_k$, where each induced graph $S_G[V_i]$ is either a triangle or a diamond,
and that the vertices of $C_G$ are denoted by $v_1,\ldots,v_k$, where each $v_i$ corresponds to the set $V_i$.
Suppose $I=\{v_1,\ldots,v_s\}$ is an independent set in $C_G$.

\begin{lemma}\label{l1}
    Let $v_i,v_j\in I$ be distinct vertices. If $x\in V_i$ and $y\in V_j$, then $d_{S_G}(x,y)\geq 3$. Moreover, equality holds if and only if there exists $i_0\in \{s+1,\dots ,k\}$ such that $S_G[V_{i_0}]$ is a triangle containing a neighbor of $x$ and a neighbor of $y$.
\end{lemma}
\begin{proof}
    Suppose that $x\in V_i$ and $y\in V_j$. Let $u_1\cdots u_{\ell}$ be an $xy$-path in $S_G$ of minimum length, where $u_1=x$ and $u_{\ell}=y$. As $I$ is independent in $C_G$, no vertex of $V_i$ is adjacent to a vertex in $V_j$. Then, $\ell\geq 3$. Hence, there exists $i_0\in \{s+1,\dots ,k\}$ such that $u_2\in V_{i_0}$. As $u_2$ has a unique neighbor outside $V_{i_0}$, it follows that $u_3\in V_{i_0}$. Since $u_2\notin V_i$ and $u_1\in V_i$, the vertex $u_2$ has a neighbor outside $V_{i_0}$. Hence, by Remark~\ref{r2}, $u_2$ is a light $3$-vertex and therefore has a unique neighbor outside $V_{i_0}$, namely $u_1$. Consequently, the next vertex on the path satisfies $u_3\in V_{i_0}$. Since $i_0\in \{s+1,\dots ,k\}$, then $i_0\neq j$. Therefore, $\ell\geq 4$ and $d_{S_G}(x,y)\geq 3$. Moreover, the equality occurs if and only if $\ell=4$, that is, $y$ is adjacent to $u_3$. By Remark \ref{r2}, $u_3$ is a light vertex. Now, as $u_2$ and $u_3$ are light vertices, it follows that $S_G[V_{i_0}]$ is a triangle. The result follows.
\end{proof}

Although vertices belonging to different blocks of an independent set are sufficiently separated, some pairs remain at distance exactly three. Since such pairs cannot simultaneously belong to a 3-packing, we encode these conflicts by introducing the following notion. For $v_i\in I$ and $x\in V_i$, a vertex $y\in V(S_G)$ is said to be \textit{$I$-perfectly close} to $x$ if there exists $v_j\in I$ such that $y\in V_j$ and $d_{S_G}(x,y)=3$.

\begin{remark}\label{r3}
    Let $v_i\in I$ and $x\in V_i$. A vertex $y$ is $I$-perfectly close to $x$ if and only if the following hold:\begin{itemize}
        \item[(i)] $x$ and $y$ are light $3$-vertices,
        \item[(ii)] There exists $v_j\in I$ such that $y\in V_j$, and 
        \item[(iii)] There exists $v_{\ell}\in V(C_G)\setminus I$ such that $S_G[V_{\ell}]$ is a triangle and both $x$ and $y$ have a neighbor (in $S_G$) in $V_{\ell}$.
    \end{itemize}
\end{remark}

The relation of being $I$-perfectly close naturally defines a conflict graph on the candidate vertices. Construct the graph $H_I$ such that $V(H_I)=\bigcup_{v_i\in I} V_i$ and $xy\in E(H_I)$ if and only if $x$ and $y$ are $I$-perfectly close.

\begin{remark}\label{r4}
    By Remark \ref{r3}, we have $\Delta(H_I)\leq 2$.
\end{remark}

\begin{lemma}\label{l2}
    Every component of $H_I$ is of order at most $3$.
\end{lemma}
\begin{proof}
  Suppose that a component $C$ of $H_I$ contains at least three vertices. Then, $C$ contains a vertex $x$ of degree at least $2$. Let $y$ and $z$ be the neighbors of $x$ in $H_I$. By Remark \ref{r3}, we may suppose that $x\in V_1$, $y\in V_3$, $z\in V_4$, $S_G[V_2]$ is a triangle, and $x$, $y$, and $z$ have neighbors (in $S_G$) in $V_2$. Now, $x$, $y$, and $z$ have no neighbors in any $V_i$, for every $i\geq 5$. Then, $C$ contains only $x$, $y$, and $z$. Indeed, the vertices $y$ and $z$ are light $3$-vertices. Since each of them has a unique neighbor outside its own block, and this neighbor already lies on the path joining $x$ to the corresponding block, neither $y$ nor $z$ can be perfectly close to a vertex belonging to another block. Hence $C$ contains no vertex other than $x$, $y$, and $z$. The result follows.
\end{proof}

Let $C_1,\dots ,C_m$ be all components of $H_I$. Our objective is to select exactly one vertex from every set $V_i$ while avoiding conflicts represented by $H_I$. Since vertices belonging to the same component of $H_I$ cannot all be chosen independently, we transform this selection problem into a bipartite matching problem. Construct the bipartite graph $B_I$ with bipartition $U:=\{u_1,\dots ,u_s\}$ and $C:=\{c_1,\dots ,c_m\}$ such that $u_ic_j\in E(B_I)$ if and only if $V_i$ contains a vertex in $C_j$.

\begin{lemma}\label{l3}
For every subset $A\subseteq U$, we have $|N_{B_I}(A)|\ge |A|.$
\end{lemma}

\begin{proof}
Let $A\subseteq U$. For each vertex $u_i\in U$, since $S_G[V_i]$ is either a triangle or a diamond, we have
$|V_i|\geq 3.$ Moreover, the sets $V_i$ are pairwise disjoint. Hence, \begin{equation}\label{eq1}
    \left|\bigcup_{u_i\in A}V_i\right|=\sum_{u_i\in A}|V_i|\geq 3|A|.
\end{equation}

By the definition of the bipartite graph $B_I$, every vertex of $\bigcup_{u_i\in A}V_i$
belongs to a component of $H_I$ corresponding to a vertex of $N_{B_I}(A)$. Therefore, 
$$\bigcup_{u_i\in A}V_i \subseteq \bigcup_{c_j\in N_{B_I}(A)}V(C_j).$$
Since every component of $H_I$ has order at most $3$ by Lemma~\ref{l2}, we obtain
\begin{equation}\label{eq2}
    \left|\bigcup_{u_i\in A}V_i\right|\leq \sum_{c_j\in N_{B_I}(A)}|V(C_j)|\leq3|N_{B_I}(A)|.
\end{equation}
Then, by (\ref{eq1}) and (\ref{eq2}), we obtain
$$3|A|\leq 3|N_{B_I}(A)|.$$
Hence, 
$$|A|\leq |N_{B_I}(A)|.$$
This completes the proof.
\end{proof}

Now, we are ready to use Hall's theorem to find a good $3$-packing in $S_G$.
\begin{lemma}\label{l4}
    There exists a $3$-packing $Y$ in $S_G$ such that $|Y\cap V_i|=1$ for all $v_i\in I$.
\end{lemma}
\begin{proof}
    By Lemma~\ref{l3} and Hall's theorem, there exists a matching $M$ saturating U. For each $u_i\in U$, let $c_j$ be the vertex of $C$ matched to $u_i$ by $M$. Since $u_ic_j\in E(B_I)$, a vertex of  $V_i$ is contained in $C_j$. Let $x_i\in V_i\cap V(C_j)$ and $Y=\{x_1,\dots ,x_s\}$. Since $M$ is a matching, no two vertices in $Y$ belong to the same component in $H_I$. Thus, the vertices in $Y$ are not $I$-perfectly close. Therefore, by Lemma \ref{l1}, $Y$ is a $3$-packing. By the definition of $Y$, we have $|Y\cap V_i|=1$ for all $v_i\in I$.
\end{proof}

\section{A $(1,1,2,3)$-packing coloring}
In this section, we prove that every claw-free subcubic graph is $(1,1,2,3)$-packing colorable. Let $G$ be a connected claw-free subcubic graph with minimum degree at least $2$ and containing a $3$-vertex such that $G\neq K_4$. We will proceed as follows:\\
\textbf{Step 1.} We find a $2$-packing $X$ in $S_G$ such that the set of triangles in $S_G-X$ corresponds to an independent set $I$ in $C_G$.\\
\textbf{Step 2.} By Lemma \ref{l4}, one can find a $3$-packing $Y$ in $S_G$ that meets every triangle in $S_G-X$.\\
\textbf{Step 3.} We prove that a choice of such a pair of $2$-packing $X$ and a $3$-packing $Y$ in $S_G$ can eliminate every odd cycle in $G$.\\
\textbf{Step 4.} As $X$ (resp., $Y$) is a $2$-packing (resp., $3$-packing) in $G$ also, the graph $G-(X\cup Y)$ can be partitioned into two independent sets $S_1$ and $S_2$. By coloring the vertices of $S_1,S_2,X,Y$ using the colors $1_a,1_b,2,3$ respectively, we obtain a $(1,1,2,3)$-packing coloring of $G$.

\begin{lemma}\label{l5}
Let $x$ and $y$ be two non-adjacent heavy vertices in $S_G$. Then, $d_{S_G}(x,y)\geq 3$.
\end{lemma}
\begin{proof}
Suppose, to the contrary, that $d_{S_G}(x,y)=2$. Then, there exists a vertex $z\in N(x) \cap N(y)$. Since $G$ is subcubic, we may write $N(x)=\{z,a,b\}$ and $N(y)=\{z,c,d\}$. Since $x$ is heavy, $x$ belongs to two triangles. Every triangle containing $x$ corresponds to an edge between two neighbors of $x$. Thus, $z$ is adjacent to $a$ or $b$. Similarly, since $y$ is heavy, we deduce that $z$ is adjacent to $c$ or $d$. Then, we get that $d(z)\ge 4$, a contradiction since $G$ is subcubic.    
\end{proof}

We now consider a set $X \subseteq V(S_G)$ satisfying the following properties:
\begin{enumerate}
    \item for every adjacent pair $(R,S)$ of triangles, exactly one of the two common heavy vertices of $R$ and $S$ belongs to $X$;
    \item $X$ is a $2$-packing;
    \item for every close pair $(R,S)$ of triangles, either $X \cap V(R\cup S)\neq \emptyset$, or every vertex of $R\cup S$ is at distance at most $2$ from some vertex of $X$.
\end{enumerate}
Such a set $X$ will be called a \textit{close $2$-packing}. Clearly, such a close 2-packing exists. For a close $2$-packing $X$, we denote by $\theta(X)$ the number of close pairs in the graph $S_G-X$. Among all close $2$-packings, let $X$ be one minimizing $\theta(X)$. 
\begin{claim}\label{l6}
$\theta(X)=0$.
\end{claim}
\begin{proof}
Suppose, to the contrary, that $\theta(X)>0$. Then, there exists a close pair $(R,S)$ in $S_G-X$. Since no vertex of $R\cup S$ belongs to $X$, it follows from the definition of a close $2$-packing that every vertex of $R\cup S$ is at distance at most $2$ from some vertex of $X$.

Let $R=abc$ and $S=a'b'c'$ such that $aa'\in E(S_G)$. Let $x$ and $u$ be the third neighbors of $b$ and $c$, respectively. Note that at least one of $x$ and $u$ belongs to $X$. Otherwise, $a$ would be at distance at least three from every vertex of $X$, contradicting the fact that every vertex of $R$ is at distance at most two from some vertex of $X$. 

Without loss of generality, assume that $x\in X$. Note that since each vertex of $R$ is light, $x$ and $u$ are both light vertices. Let $R_x=xyz$ be the triangle containing $x$. Since $X$ is a $2$-packing, neither $y$ nor $z$ belongs to $X$. Let $y'$ and $z'$ be the third neighbors of $y$ and $z$, respectively. Note that if $y'=z'$, then $y$ and $z$ are heavy vertices. Hence, we have $y\in X$ or $z\in X$, a contradiction. Then, $y'\neq z'$. Thus, we deduce that  $(R_{y'},R_x)$ and $(R_{z'},R_x)$ are close pairs. We now consider two cases.

\medskip
\noindent
\textbf{Case 1.} $u\notin X$.\\
Here, $a$ is at distance at least $3$ from every vertex of $X\setminus\{x\}$. 
 Suppose that $X\cap R_{y'}\neq\emptyset$ and $X\cap R_{z'}\neq\emptyset$. Let
$X'=(X\setminus\{x\})\cup\{a\}$.
Since $a$ is at distance at least $3$ from every vertex of $X\setminus\{x\}$, the set $X'$ is a close $2$-packing. Moreover, the pair $(R,S)$ is no longer a close pair in $G-X'$ and no new close pair is created. Hence, $\theta(X')<\theta(X),$ contradicting the minimality of $\theta(X)$. Thus, $X\cap R_{y'}=\emptyset$ or $X\cap R_{z'}=\emptyset$.
Without loss of generality, assume that $X\cap R_{y'}=\emptyset$. Since $x\in X$ and $X$ is a $2$-packing, we have $z'\notin X$. Therefore, $y$ is at distance at least three from every vertex of $X\setminus\{x\}$. Let $X'=(X\setminus\{x\})\cup\{y,a\}$.
Then, $X'$ is a close $2$-packing and $(R,S)$ is no longer a close pair in $S_G-X'$. Since no close pair is created, we have $\theta(X')<\theta(X)$, a contradiction.

\medskip
\noindent
\textbf{Case 2.} $u\in X$.
\\
Let $R_u=uvw$ be the triangle containing $u$. Let $v'$ and $w'$ be the third neighbors of $v$ and $w$, respectively. Similarly as above, we deduce that $(R_{v'},R_u)$ and $(R_{w'},R_u)$ are close pairs. 
Suppose that $X\cap R_{\alpha}\neq\emptyset $ for every $\alpha\in\{y',z',v',w'\}$. Let $X'=(X\setminus\{x,u\})\cup\{a\}$.
Then, $X'$ is a close $2$-packing and the pair $(R,S)$ is no longer a close pair in $S_G-X'$. Since $X\cap R_{\alpha}\neq\emptyset $ for every $\alpha\in\{y',z',v',w'\}$, no new close pair is created. Hence, we have $\theta(X')<\theta(X)$, a contradiction.
Otherwise, there exists $\alpha \in\{y',z',v',w'\}$ such that $X\cap R_{\alpha}=\emptyset $.  Without loss of generality, suppose that $\alpha=y'$. Since $x\in X$ and $X$ is a $2$-packing, we have $z'\notin X$. Hence, $y$ is at distance at least three from every vertex of $X\setminus\{x\}$.

Suppose that $R_{v'}\cap X =\emptyset$ or $R_{w'}\cap X =\emptyset$. Without loss of generality, suppose that $R_{v'}\cap X =\emptyset$. Since $u\in X$ and $X$ is a $2$-packing, we have $w'\notin X$. Then, $v$ is at distance at least three from each vertex in $X\setminus \{u\}$.
Let $X'=(X\setminus\{x,u\})\cup\{a,y,v\}.$
Then, $X'$ is a close $2$-packing and $(R,S)$ is no longer a close pair in $S_G-X'$. Since no new close pair is created, we deduce that $\theta(X')<\theta(X)$, a contradiction.

Thus, $R_{v'}\cap X \neq \emptyset$ and $R_{w'}\cap X \neq \emptyset$. Let $X'=(X\setminus\{x,u\})\cup\{a,y\}.$
Then, $X'$ is a close $2$-packing and $(R,S)$ is no longer a close pair in $S_G-X'$. Since no new close pair is created, we deduce that $\theta(X')<\theta(X)$, a contradiction.    
\end{proof}

By Claim \ref{l6}, the triangles in $S_G-X$ are not connected by an edge. Then, the set $I\subseteq V(C_G)$ that corresponds to the triangles in $S_G-X$ is independent. Now, by Lemma \ref{l4}, there exists a $3$-packing $Y$ in $S_G$ that meets every triangle in $S_G-X$.

Let $A$ and $B$ be a $2$-packing and $3$-packing in $G$, respectively. Suppose that $A$ and $B$ are disjoint. If $G-(A\cup B)$ is triangle-free, then the pair $(A,B)$ is said to be a \textit{triangle cover pair}. By the above, there exists a 2-packing $X$ and a 3-packing $Y$ such that $X,Y\subseteq V(S_G)$, $X$ and $Y$ are disjoint, and $S_G-(X\cup Y)$  is triangle-free. Thus, by the construction of $S_G$, we can notice that $G-(X\cup Y)$ is triangle-free and so a triangle cover pair exists in $G$. For a triangle cover pair $(A,B)$ of $G$, let $\phi(A,B)$ be the number of odd cycles in $G-(A\cup B)$. Among all triangle cover pairs, suppose that $(A,B)$ is chosen with minimum $\phi(A,B)$. Throughout the proof, every modification of $(A,B)$ will preserve the defining properties of a triangle cover pair while strictly decreasing this number, yielding the desired contradiction.

\begin{lemma}\label{l7}
    $\phi(A,B)=0$.
\end{lemma}
\begin{proof}
Suppose, to the contrary, that $G' := G- (A\cup B)$ contains an odd cycle $C$. Since $G'$ contains no triangles, we have $l(C) \ge 5$. Set $C = x_1 x_2 \ldots x_{2t+1}$, $t\geq 2$.  

Note that each vertex $x$ of $C$ is at distance less than 3 from some vertex in $A$. Otherwise, suppose that there exists a vertex $x$ in $C$ such that $x$ is at distance at least $3$ from each vertex in $A$. Let $A'=A\cup \{x\}$. So, $(A', B)$ is a triangle cover pair with $\phi (A',B) < \phi (A,B)$, contradicting the minimality of $\phi(A,B)$. Similarly, each vertex $x$ of $C$ is at distance less than $4$ from some vertex in $B$. 

Since $G'$ contains no triangles and since each $3$-vertex in $G$ is contained in a triangle, we deduce that each $3$-vertex $x_i$ in $C$ is contained in a triangle of the form $x_ix_{i+1}a$ or $x_ix_{i-1}a$ where $a\in A\cup B$.

\begin{claim}
    Every $3$-vertex in $C$ is contained in exactly one triangle.
\end{claim}
\begin{proof}
    Suppose, to the contrary, that there exists a $3$-vertex $x_i\in V(C)$ such that $x_i$ is contained in two triangles. Without loss of generality, assume that $i=1$. Then, the two triangles containing $x_1$ are of the form $x_{2t+1}x_1u$ and $x_1x_2u$ for some $u\in A\cup B$. Suppose that $u\in B$. Then, $x_1$ is at distance at least three from each vertex in $A$, a contradiction. Thus, $u\in A$. Either $x_3$ or $x_{2t}$ has a neighbor in $B$ since otherwise $x_1$ is at distance at least four from each vertex in $B$. Without loss of generality, suppose that $x_3$ has a neighbor in $B$ say $v_1$. Note that $x_3$ lies in a unique triangle which is $x_3x_4v_1$. Thus, $v_1$ is a light vertex.
    
    Let $v'_1$ be the third neighbor of $v_1$, if it exists. Let $y_1,z_1$ be the neighbors of $v'_1$ if they exist. Note that if $v'_1$ is a $3$-vertex, then $v'_1$ is a light vertex. Since $B$ is a $3$-packing, neither of $y_1$ and $z_1$ belongs to $B$. We will study two cases:

\medskip
\noindent
\textbf{Case 1:} $v'_1\in A$.

Since $v_1\in B$ and $B$ is a $3$-packing, we deduce that $x_5$ has no neighbor in $B$. Hence, $x_3$ is at distance at least four from each vertex in $B\setminus \{v_1\}$. Let $B'=(B\setminus \{v_1\})\cup \{x_3\}$. So, $(A,B')$ is a triangle cover pair. Since any odd cycle containing $v_1$ must pass through $x_3$ or $v_1'$, we deduce that $v_1$ is not contained in any odd cycle in $G\setminus(A\cup B')$. Thus, $(A,B')$ is a triangle cover pair with $\phi (A,B') < \phi (A,B)$, a contradiction.

\medskip
\noindent
\textbf{Case 2:} $v'_1\notin A$.

Suppose that $x_1$ is at distance at least four from each vertex in $B\setminus \{v_1\}$. Then, $x_{2t}$ has no neighbors in $B$. Let $A'=(A\setminus \{u\})\cup \{ x_3\}$ and $B'=(B\setminus \{v_1\})\cup \{u\}$. Since all $3$-vertices in $C$ have a neighbor in $A\cup B$, we deduce that each cycle containing $v_1$ in $(G- (A\cup B))\cup \{v_1\}$ must pass through $x_3$. Thus, $v_1$ is not contained in any odd cycle in $G- (A'\cup B')$. Then, $(A',B')$ is a triangle cover pair with $\phi(A',B')<\phi (A,B)$, a contradiction. Hence, $x_{2t}$ has a neighbor in $B$. Then, $x_{2t+1}$ is at distance at least three from each vertex in $A\setminus\{u\}$. Let $z$ be the third neighbor of $x_{2t}$ where $z\in B$. Note that if $C$ is of length $5$, then $z=v_1$. Then, we get that $x_1$ is at distance at least $4$ from each vertex in $B\setminus \{v_1\}$, a contradiction as above. Hence, $l(C) >5$ and so $d(x_3,x_{2t+1})>2$. Let $A'=(A\setminus\{u\})\cup \{x_3, x_{2t+1}\}$ and 
$B'=(B\setminus\{v_1\})\cup \{x_2\}$. Since each cycle containing $u$ must contain $x_2$ or $x_{2t+1}$, we deduce that $u$ is not contained in any odd cycle in $G-(A'\cup B')$. Moreover, since all $3$-vertices of $C$ have a neighbor in $A\cup B$, we deduce that each cycle containing $v_1$ in $(G- (A\cup B)) \cup \{v_1\}$ must pass through $x_3$. So, $v_1$ is not contained in any odd cycle in $G- (A'\cup B')$. Then, $(A',B')$ is a triangle cover pair with $\phi(A',B')<\phi (A,B)$, a contradiction. 
\end{proof}

Now, each $3$-vertex in $C$ is contained in exactly one triangle and this triangle intersects $C$ at two vertices. Thus, we deduce that the number of $3$-vertices in $C$ is even. Therefore, $C$ contains a $2$-vertex since $C$ is odd. Without loss of generality, suppose that $x_1$ is a $2$-vertex. As $x_1$ is at distance at most $2$ from a vertex in $A$, we may suppose that $x_2$ is adjacent to a vertex $v_1\in A$. Note that $v_1$ is also adjacent to $x_3$ and $v_1$ is not adjacent to $x_4$. Let $v'_1$ be the neighbor of $v_1$ other than $x_2$ and $x_3$, if it exists. 

Suppose that $v'_1\in B$. Hence, $x_2$ is at distance at least $3$ from each vertex in $A\setminus \{v_1\}$. Let $A'=(A\setminus \{v_1\})\cup \{x_2\}$. So, $(A',B)$ is a triangle cover pair. Since any odd cycle containing $v_1$ must pass through $x_2$ or $v_1'$, we deduce that $v_1$ is not contained in any odd cycle in $G-(A'\cup B)$. Thus, $(A',B)$ is a triangle cover pair with $\phi (A',B) < \phi (A,B)$, a contradiction.

Now, assume that $v'_1\notin B$. Let $z$ be a vertex in $B$ that is at a distance of at most $3$ from $x_1$. If $z$ is not adjacent to $x_{2t+1}$, then $z$ is adjacent to $x_{2t}$. Then, since each $3$-vertex is contained in a triangle, we deduce that $zx_{2t}x_{2t-1}$ is a triangle and $x_{2t+1}$ is a $2$-vertex. So, $x_{2t+1}$ is at distance more than $2$ from each vertex in $A$, a contradiction. Then, $z$ is adjacent to $x_{2t+1}$ and $zx_{2t+1}x_{2t}$ is a triangle. As $x_{2t+1}$ is at distance less than $3$ from a vertex in $A$, the vertex $z$ is adjacent to a vertex $z'\in A$. Let $B'=(B\setminus \{z\})\cup \{x_{2t+1}\}$. Since any odd cycle containing $z$ must pass through $x_{2t+1}$ or $z'$, we deduce that $z$ is not contained in any odd cycle in $G-(A\cup B')$. Thus, $\phi(A,B')<\phi(A,B)$, a contradiction.
\end{proof}  

\begin{theorem}\label{t3}
    Every claw-free subcubic graph is $(1,1,2,3)$-packing colorable.
\end{theorem}
\begin{proof}
    On the contrary, suppose that $G$ is a counterexample of minimum order. We may assume that $G$ is connected. Clearly $G$ is neither a cycle nor $K_4$. Suppose to the contrary that $G$ contains a leaf $u$. Consider a $(1,1,2,3)$-packing coloring of $G-u$ using the colors $1_a,1_b,2,3$. Then, either $1_a$ or $1_b$ is not assigned to the neighbor of $u$, say $1_a$. By coloring $u$ using the color $1_a$ we obtain a $(1,1,2,3)$-packing coloring of $G$, a contradiction. Then, we may assume that $\delta(G)\geq 2$. Moreover, as $G$ is not a cycle, $G$ contains a $3$-vertex. By Lemma \ref{l7}, there exists a triangle cover pair $(X,Y)$ such that $\phi(X,Y)=0$. Let $\{I_1,I_2\}$ be a partition of $V(G)\setminus (X\cup Y)$ such that $I_1$ and $I_2$ are independent. Color the vertices of $I_1$, $I_2$, $X$, and $Y$ with the colors $1_a$, $1_b$, $2$, and $3$, respectively, to obtain a $(1,1,2,3)$-packing coloring of $G$.
\end{proof}

\section{A $(1,1,3,3,3)$-packing coloring}
We now adapt the approach developed in the previous section to obtain a $(1,1,3,3,3)$-packing coloring. The main idea is to replace the single independent set of the core graph by a proper $3$-coloring of the core graph. Hall’s theorem is then applied independently to each color class to construct three disjoint $3$-packings of the skeleton. Finally, we modify these packings until all triangles and odd cycles disappear.

Let $G$ be a connected claw-free subcubic graph with minimum degree at least $2$ and containing a $3$-vertex and such that $G\notin \{K_4,\mathcal{H}\}$ (see Figure \ref{figure_1}).

\begin{proposition}\label{prop1}
    If $C_G=K_4$, then $G$ is $(1,1,3,3,3)$-packing colorable.
\end{proposition}
\begin{proof}
    Suppose that $C_G=K_4$ and let $V_1,V_2,V_3,V_4$ be the partition of the vertex set of $S_G$ such that $S_G[V_i]$ is either a triangle or a diamond, for all $i\in \{1,2,3,4\}$. Since a vertex representing a diamond in $C_G$ has at most two neighbors, we can deduce that $S_G[V_i]$ is a triangle, for all $i\in \{1,2,3,4\}$. Hence, $S_G=\mathcal{H}$. But, $G\neq \mathcal{H}$. Thus, $G$ is a graph obtained by subdividing at least one edge of $\mathcal{H}$. Without loss of generality, suppose that the edge between $V_1$ and $V_2$ is subdivided. Then, one can choose vertices $x\in V_1$ and $y\in V_2$ such that $dist_G(x,y)\geq 4$. Let $u\in V_3$ and $v\in V_4$ such that $uv\in E(S_G)$. Observe that $G-\{x,y,u,v\}$ is a path. Then, we can color $x,y,u,v$ by $3_a,3_a,3_b,3_c$ respectively and the remaining vertices by $1_a$ and $1_b$ in order to obtain a $(1,1,3,3,3)$-packing coloring of $G$.
\end{proof}

From now on, we may assume that $C_G\neq K_4$.
\begin{lemma}\label{l8}
    $\chi(C_G)\leq 3$.
\end{lemma}
\begin{proof}
   By Remark \ref{r2}, we have $\Delta(C_G)\leq 3$. Moreover, since $G$ is connected, so are $S_G$ and $C_G$. By Brooks' theorem, and since $C_G\neq K_4$ the result follows.
\end{proof}

Let $\{\mathcal{S}_1,\mathcal{S}_2,\mathcal{S}_3\}$ be a partition of $V(C_G)$ into independent sets.

\begin{lemma}\label{l9}
    There exist three disjoint $3$-packings $S_1$, $S_2$, and $S_3$ in $S_G$ such that $|(S_1\cup S_2\cup S_3)\cap V_j|=1$ for all $j\in \{1,\dots ,k\}$.
\end{lemma}
\begin{proof}
    Let $i\in \{1,2,3\}$. As $\mathcal{S}_i$ is an independent set in $C_G$, by Lemma \ref{l4}, there exists a $3$-packing $S_i$ in $S_G$ such that $|S_i\cap V_j|=1$ for all $v_j\in S_i$. Now, since $\{\mathcal{S}_1,\mathcal{S}_2,\mathcal{S}_3\}$ form a partition of $V(C_G)$, it follows that, for every $j\in \{1,\dots ,k\}$, there exists $i_j\in \{1,2,3\}$ such that $|S_{i_j}\cap V_j|=1$.
    Since the three color classes partition $V(C_G)$, every $V_j$ belongs to exactly one $\mathcal S_i$, so every selected vertex comes from a different family of $V_j$.
\end{proof}

Let $S_1,S_2,S_3$ be three disjoint $3$-packings in $S_G$ such that $|(S_1\cup S_2\cup S_3)\cap V_j|\geq 1$ for all $j\in \{1,\dots ,k\}$. Let $\varphi(S_1,S_2,S_3)$ be the number of triangles in $S_G-(S_1\cup S_2\cup S_3)$. Among all choices of such three $3$-packings, suppose that $S_1$, $S_2$, and $S_3$ are chosen such that $\varphi(S_1,S_2,S_3)$ is minimum.

\begin{lemma}\label{l10}
    $\varphi(S_1,S_2,S_3)=0$.
\end{lemma}
\begin{proof}
    Suppose that $S_G-(S_1\cup S_2\cup S_3)$ contains a triangle. Then, there exist $i_0\in \{1,2,3\}$ and $j_0\in \{1,\dots ,k\}$ such that $S_{i_0}\cap V_{j_0}=\{x_{i_0}\}$, $x_{i_0}$ is a light vertex, and $S_G[V_{j_0}]$ is a diamond. Without loss of generality, suppose that $i_0=j_0=1$. Note that every heavy vertex has no $S_1$-perfectly close vertex. Let $x'_1\in V_1$ be a heavy vertex and $S'_1=(S_1\setminus \{x_1\})\cup \{x'_1\}$. The set of three $3$-packings $\{S'_1,S_2,S_3\}$ meets every $V_i$ such that $\varphi(S'_1,S_2,S_3)<\varphi(S_1,S_2,S_3)$, a contradiction.
\end{proof}

By the definition of $S_G$ and Lemma \ref{l10}, there exist three disjoint $3$-packings $S_1,S_2,S_3$ in $G$ such that $|(S_1\cup S_2\cup S_3)\cap V_j|\geq 1$ for all $j\in \{1,\dots ,k\}$, and $\varphi(S_1,S_2,S_3)=0$. Let $\phi(S_1,S_2,S_3)$ be the number of odd cycles in $G-(S_1\cup S_2\cup S_3)$. Among all choices of such three $3$-packings, suppose that $S_1$, $S_2$, and $S_3$ are chosen such that $\phi(S_1,S_2,S_3)$ is minimum.

\begin{lemma}\label{l11}
    $\phi(S_1,S_2,S_3)=0$.
\end{lemma}
\begin{proof}
    Suppose, to the contrary, that $G' := G- (S_1\cup S_2\cup S_3)$ contains an odd cycle $C$. Since $G'$ contains no triangles, we have $l(C) \ge 5$. Set $C = x_1 x_2 \ldots x_{2t+1}$, $t\geq 2$.  

    Note that each vertex $x$ of $C$ is at distance less than $4$ from some vertex in $S_i$, for every $i\in \{1,2,3\}$.

    Since $G'$ contains no triangles and since each $3$-vertex in $G$ is contained in a triangle, we deduce that each $3$-vertex $x_i$ in $C$ is contained in a triangle of the form $x_ix_{i+1}a$ or $x_ix_{i-1}a$ where $a\in S_1\cup S_2\cup S_3$.

    \begin{claim}
    Every $3$-vertex in $C$ is contained in exactly one triangle.
    \end{claim}
    \begin{proof}
    Suppose, to the contrary, that there exists a $3$-vertex $x_i\in V(C)$ such that $x_i$ is contained in two triangles. Without loss of generality, assume that $i=1$. Then, the two triangles containing $x_1$ are of the form $x_{2t+1}x_1u$ and $x_1x_2u$ for some $u\in S_1\cup S_2\cup S_3$. Without loss of generality, suppose that $u\in S_1$. As $x_1$ is at distance less than $4$ from some vertex in $S_i$, for every $i\in \{2,3\}$, it follows that $x_3$ (resp., $x_{2t}$) is adjacent to a vertex $v_1\in S_{i_1}$ (resp., $v_2\in S_{i_2}$); where $\{i_1,i_2\}=\{2,3\}$. Without loss of generality, suppose that $i_1=2$ and $i_2=3$. Note that $x_3x_4v_1$ and $x_{2t}x_{2t-1}v_2$ are both triangles. Now, since $x_2$ is at distance less than $4$ from a vertex in $S_3$, it follows that $v_1$ is adjacent to a vertex $v'_1\in S_3$. Let $S'_2=(S_2\setminus \{v_1\})\cup \{x_3\}$. Since any odd cycle containing $v_1$ must pass through $x_3$ or $v_1'$, we deduce that $v_1$ is not contained in any odd cycle in $G\setminus(S_1\cup S'_2\cup S_3)$. Thus, $\phi(S_1,S'_2,S_3)<\phi(S_1,S_2,S_3)$, a contradiction.
    \end{proof}

    Now, each $3$-vertex in $C$ is contained in exactly one triangle and this triangle intersects $C$ at two vertices. Thus, we deduce that the number of $3$-vertices in $C$ is even. Therefore, $C$ contains a $2$-vertex. Without loss of generality, suppose that $x_1$ is a $2$-vertex. As $x_1$ is at distance at most $3$ from a vertex in $S_i$, for every $i\in \{1,2,3\}$, we may suppose the following: (1) $x_2$ is adjacent to a vertex $v_1\in S_1$, (2) for some $\ell\in \{2t,2t+1\}$, $x_{\ell}$ is adjacent to a vertex in $S_2$, say $v_2$, and (3) $v_1$ or $v_2$ is adjacent to a vertex in $S_3$. Without loss of generality, suppose $v_1$ is adjacent to $v'_1\in S_3$. Let $S'_1=(S_1\setminus \{v_1\})\cup \{x_2\}$. Again, we have $\phi(S'_1,S_2,S_3)<\phi(S_1,S_2,S_3)$, which is a contradiction.  
\end{proof}

\begin{theorem}\label{t4}
    Every connected claw-free subcubic graph, except $\mathcal{H}$, is $(1,1,3,3,3)$-packing colorable.
\end{theorem}
\begin{proof}
    On the contrary, suppose that $G$ is a counterexample of minimum order. Clearly $G$ is neither a cycle nor $K_4$. Suppose to the contrary that $G$ contains a leaf $u$. Since $\mathcal{H}$ is cubic, $G-u\neq \mathcal{H}$. Consider a $(1,1,3,3,3)$-packing coloring of $G-u$ using the colors $1_a,1_b,3_a,3_b,3_c$. Then, either $1_a$ or $1_b$ is not assigned to the neighbor of $u$, say $1_a$. By coloring $u$ using the color $1_a$ we obtain a $(1,1,3,3,3)$-packing coloring of $G$, a contradiction. Then, we may assume that $\delta(G)\geq 2$. Moreover, as $G$ is not a cycle, $G$ contains a $3$-vertex. By Proposition \ref{prop1}, it follows that $C_G\neq K_4$. By Lemma \ref{l11}, there exist three disjoint $3$-packings $S_1,S_2,S_3$ such that $\phi(S_1,S_2,S_3)=0$. Let $\{I_1,I_2\}$ be a partition of $V(G)\setminus (S_1\cup S_2\cup S_3)$ such that $I_1$ and $I_2$ are independent. Color the vertices of $I_1$, $I_2$, $S_1$, $S_2$, and $S_3$ with the colors $1_a$, $1_b$, $3_a$, $3_b$, and $3_c$, respectively, to obtain a $(1,1,3,3,3)$-packing coloring of $G$.
\end{proof}

\section{Conclusion and open problems}
We proved that every claw-free subcubic graph is $(1,1,2,3)$-packing colorable. The result is sharp in the sense that there are infinitely many claw-free subcubic graphs that are neither $(1,1,3,3)$- nor $(1,2,2,3)$-packing colorable.

\begin{proposition}\label{sharp1}
    The graph $\mathcal{A}_1$ (see Figure \ref{figure_2}) is neither $(1,1,3,3)$- nor $(1,2,2,3)$-packing colorable. Moreover, there are infinitely many connected claw-free subcubic graphs that contain $\mathcal{A}_1$ as a subgraph.
\end{proposition}
\begin{proof}
    The graph $\mathcal{A}_1$ contains three vertex-disjoint triangles. Then, in any $(1,1,3,3)$-packing coloring of $\mathcal{A}_1$, there are at least three  vertices colored by a color $3$. This is impossible since the diameter of $\mathcal{A}_1$ is $3$. 

    Suppose that $\mathcal{A}_{1}$ has a $(1,2,2,3)$-packing coloring. Since the diameter of $\mathcal{A}_{1}$ is $3$, then at most one vertex can be colored by $3$. If three of the vertices of the cycle of order $6$ are colored by $1$, then no more vertices are colored by $1$. In this case, at most four of the remaining vertices can be colored by the two colors $2$, a contradiction since the order of $\mathcal{A}_1$ is $9$. Otherwise, at most two vertices of the cycle of order $6$ are colored by $1$, then at most five of the remaining vertices can be colored by the two colors $2$ and the color $1$, a contradiction.

    Finally, since $\mathcal{A}_1$ contains a $2$-vertex, there are infinitely many claw-free subcubic graphs that contain $\mathcal{A}_1$ as a subgraph.
\end{proof}

Furthermore, the $(1,1,2,3)$-packing colorability is sharp in the sense that there is a claw-free subcubic graph that is not $(1,1,2,4)$-packing colorable.

\begin{proposition}
    The graph $\mathcal{A}_2$ (see Figure \ref{figure_2}) is not $(1,1,2,4)$-packing colorable.
\end{proposition}
\begin{proof}
    Since $\mathcal{A}_2$ contains four vertex-disjoint triangles, at least four vertices should be colored using the colors $2$ and $4$ in any $(1,1,2,4)$-packing coloring. But the diameter of $\mathcal{A}_2$ is $4$, then at most one vertex can be colored by the color $4$. Moreover, we cannot color three vertices by the color $2$, a contradiction.
\end{proof}
\begin{figure}[h!]
\centering
\begin{tikzpicture}[
    node/.style={circle, draw, fill=gray!10, minimum size=1mm},
    edge/.style={thick},
    green edge/.style={thick, draw=green},
    red edge/.style={thick, draw=red},
    black edge/.style={thick, draw=black},
    dotted edge/.style={thick, dotted},
    weight/.style={font=\bfseries\small}
]

\node[node] (1) at (0,0) {};
\node[node] (2) at (1,0) {};
\node[node] (3) at (2,0) {};
\node[node] (4) at (3,0) {};
\node[node] (5) at (2.5,1) {};
\node[node] (6) at (2,2) {};
\node[node] (7) at (1.5,3) {};
\node[node] (8) at (1,2) {};
\node[node] (9) at (0.5,1) {};

\draw (1) -- (2);
\draw (2) -- (3);
\draw (3) -- (4);
\draw (4) -- (5);
\draw (5) -- (6);
\draw (6) -- (7);
\draw (7) -- (8);
\draw (8) -- (9);
\draw (9) -- (1);
\draw (9) -- (2);
\draw (3) -- (5);
\draw (6) -- (8);

\node at (1.5, -1) {$\mathcal{A}_1$};


\node[node] (1) at (5.5,0) {};
\node[node] (2) at (5,1) {};
\node[node] (3) at (6,1) {};
\node[node] (4) at (5,2) {};
\node[node] (5) at (6,2) {};
\node[node] (6) at (5.5,3) {};

\node[node] (7) at (7.5,0) {};
\node[node] (8) at (7,1) {};
\node[node] (9) at (8,1) {};
\node[node] (10) at (7,2) {};
\node[node] (11) at (8,2) {};
\node[node] (12) at (7.5,3) {};

\draw (1) -- (2);
\draw (2) -- (3);
\draw (1) -- (3);
\draw (2) -- (4);
\draw (3) -- (5);
\draw (4) -- (5);
\draw (4) -- (6);
\draw (5) -- (6);
\draw (1) -- (7);
\draw (7) -- (8);
\draw (7) -- (9);
\draw (8) -- (9);
\draw (8) -- (10);
\draw (9) -- (11);
\draw (10) -- (11);
\draw (10) -- (12);
\draw (11) -- (12);
\draw (6) -- (12);

\node at (6.5, -1) {$\mathcal{A}_2$};

\end{tikzpicture}
\caption{On the left, a claw-free subcubic graph that is neither $(1,1,3,3)$-packing colorable nor $(1,2,2,3)$-packing colorable. On the right, a claw-free subcubic graph that is not $(1,1,2,4)$-packing colorable.}\label{figure_2}
\end{figure}
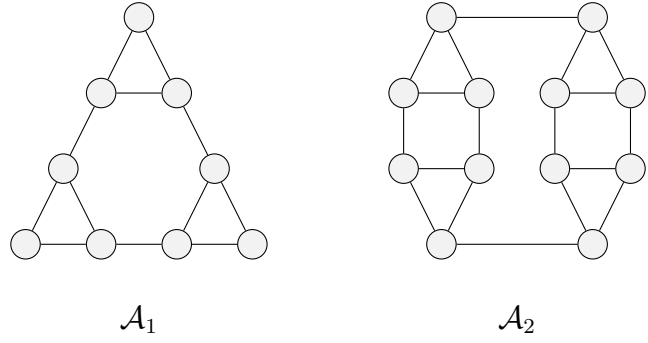

On the other hand, we proved that every connected claw-free subcubic graph, except $\mathcal{H}$, is $(1,1,3,3,3)$-packing colorable. The result is sharp in the sense that there are infinitely many connected claw-free subcubic graphs that are not $(1,1,3,3)$-packing colorable.

Moreover, there exist connected claw-free subcubic graphs that are not $(1,1,3,3,5)$-packing colorable. For instance, one can obtain such a graph from $\mathcal{A}_1$ by attaching a triangle to each $2$-vertex of $\mathcal{A}_1$, and then joining these three triangles cyclically by edges. However, we were unable to find a counterexample to $(1,1,3,3,4)$-packing colorability. This naturally leads to the following question.

\begin{problem}
Is every connected claw-free subcubic graph, distinct from $\mathcal{H}$, $(1,1,3,3,4)$-packing colorable?
\end{problem}

Moreover, our result verifies the conjectured $(1,1,2,3)$-packing colorability for claw-free subcubic graphs. The original problem of Gastineau and Togni, however, remains widely open.

\begin{problem}\cite{16}
Is every subcubic graph, except the Petersen graph, $(1,1,2,3)$-packing colorable?
\end{problem}

In another direction, Gastineau and Togni \cite{16} asked whether every subcubic graph other than the Petersen graph is $(1,2,2,2,2,2)$-packing colorable. This problem has subsequently been stated as a conjecture in several papers. El Zein and Mortada \cite{AM2} further asked whether every subcubic graph in which every $3$-vertex lies on a cycle of length at most four is $(1,2,2,2,2)$-packing colorable. Motivated by this question, we propose the following conjecture for claw-free subcubic graphs.

\begin{conjecture}
    Every claw-free subcubic graph is $(1,2,2,2,2)$-packing colorable.
\end{conjecture}

\end{document}